\documentclass[11pt,A4paper]{article}

\usepackage[left=32mm,right=32mm,top=30mm,bottom=32mm]{geometry}
\usepackage{labelfig}
\usepackage{epsfig}
\usepackage{epstopdf}

\usepackage{xfrac}

\usepackage{color}
\usepackage{amsthm,amsmath,amssymb}
\usepackage{booktabs}
\usepackage{mathpazo}
\usepackage{microtype}
\usepackage{overpic}
\usepackage{bm}
\usepackage{sectsty}
\usepackage[
	pdftitle={PDFTitle},
	pdfauthor={Hugo Parlier},
	ocgcolorlinks,
	linkcolor=linkred,
	citecolor=linkred,
	urlcolor=linkblue]
{hyperref}

\definecolor{linkred}{rgb}{0.48,0.1,0.05}
\definecolor{linkblue}{RGB}{16, 78, 139}

\usepackage[hang,flushmargin]{footmisc}
\usepackage{enumitem}
\usepackage{titlesec}
	\titlespacing{\section}{0pt}{12pt}{0pt}
	\titlespacing{\subsection}{0pt}{6pt}{0pt}
	
\titlelabel{\thetitle.\quad}

\makeatletter 

\long\def\@footnotetext#1{%
\H@@footnotetext{%
\ifHy@nesting 
\hyper@@anchor{\@currentHref}{#1}%
\else 
\Hy@raisedlink{\hyper@@anchor{\@currentHref}{\relax}}#1%
\fi 
}}

\def\@footnotemark{%
\leavevmode 
\ifhmode\edef\@x@sf{\the\spacefactor}\nobreak\fi 
\H@refstepcounter{Hfootnote}%
\hyper@makecurrent{Hfootnote}%
\hyper@linkstart{link}{\@currentHref}%
\@makefnmark 
\hyper@linkend 
\ifhmode\spacefactor\@x@sf\fi 
\relax 
}%

\ifFN@multiplefootnote%
\renewcommand*\@footnotemark{%
\leavevmode 
\ifhmode 
\edef\@x@sf{\the\spacefactor}%
\FN@mf@check 
\nobreak 
\fi 
\H@refstepcounter{Hfootnote}%
\hyper@makecurrent{Hfootnote}%
\hyper@linkstart{link}{\@currentHref}%
\@makefnmark 
\hyper@linkend 
\ifFN@pp@towrite 
\FN@pp@writetemp 
\FN@pp@towritefalse 
\fi 
\FN@mf@prepare 
\ifhmode\spacefactor\@x@sf\fi 
\relax%
}%
\fi 

\makeatother 

\theoremstyle{plain}
\newtheorem{theorem}{Theorem}[section]
\newtheorem{proposition}[theorem]{Proposition}
\newtheorem{lemma}[theorem]{Lemma}

\theoremstyle{definition}
\newtheorem{definition}[theorem]{Definition}

\newcommand{\Hyp}{{\mathbb H}}

\newcommand{\id}{{\mathrm{id}}}

\newcommand{\C}{{\mathbb C}}

\newcommand{\BB}{{\mathcal B}}

\newcommand{\M}{{\mathcal M}}

\newcommand{\arccosh}{{\,\rm arccosh}}

\sectionfont{\large \sc}
\subsectionfont{\normalsize}

\long\def\symbolfootnote[#1]#2{\begingroup%
\def\thefootnote{\fnsymbol{footnote}}\footnote[#1]{#2}\endgroup}

\def\blfootnote{\xdef\@thefnmark{}\@footnotetext}

\begin{document}

{\Large \bfseries \sc Connecting $p$-gonal loci in the compactification\\ of moduli space}

{\bfseries Antonio F. Costa, 
Milagros Izquierdo,
Hugo Parlier\symbolfootnote[2]{\normalsize All authors partially supported by Ministerio de Ciencia e Innovacion grant number MTM2011-23092 and third author supported by Swiss National Science Foundation grant number PP00P2\textunderscore 128557
}}

{\em Abstract:  } Consider the moduli space $\mathcal{M}_{g}$ of Riemann surfaces of genus $g\geq 2$ and its Deligne-Munford compactification $\overline{\mathcal{M}_{g}}$. We are interested in the branch locus ${\mathcal{B}_{g}}$ for $g>2$, i.e., the subset of $\mathcal{M}_{g}$ consisting of surfaces with automorphisms. It is well-known that the set of hyperelliptic surfaces (the hyperelliptic locus) is connected in $\mathcal{M}_{g}$ but the set of (cyclic) trigonal surfaces is not. By contrast, we show that for $g\geq 5$ the set of (cyclic) trigonal surfaces is connected in $\overline{\mathcal{M}_{g}}$. To do so we exhibit an explicit nodal surface that lies in the completion of every equisymmetric set of $3$-gonal Riemann surfaces. For $p>3$ the connectivity of the $p$-gonal loci becomes more involved. We show that for $p\geq 11$ prime and genus $g=p-1$ there are one-dimensional strata of cyclic $p$-gonal surfaces that are completely isolated in the completion $\overline{\mathcal{B}_{g}}$ of the branch locus in $\overline{\mathcal{M}_{g}}$.


\vspace{1cm}

\section{Introduction}

For $g\geq2$, moduli space $\mathcal{M}_{g}$ is
the set of conformal structures that one can put on a closed surface of genus $g$. As a set it admits many structures and can naturally be given an
orbifold structure where the orbifold points correspond exactly to those
surfaces with conformal automorphisms. One way of seeing this structure is by
seeing moduli space through the eyes of Teichm\"{u}ller theory.
Teichm\"{u}ller space is the deformation space of marked conformal structures
and is diffeomorphic to $\mathbb{R}^{6g-6}=\mathbb{C}^{3g-3}$. From this,
moduli space can be seen as the quotient of Teichm\"{u}ller space by the
mapping class group, i.e, the group of homeomorphisms of the surface up to
isotopy, which naturally acts on Teichm\"{u}ller space via the marking and as
such the mapping class group is the orbifold fundamental group of the moduli
space. The orbifold points of a moduli space appear when the surfaces in question have self-isometries and the groups of
self-isometries correspond to the finite subgroups of the mapping class group. (In the genus $2$ case this is not strictly correct because every surface is hyperelliptic, so orbifold points correspond to surfaces with additional self-isometries.)
An important step in understanding the topology, and in general, the
structures that moduli space carries, is in understanding these particular
points. The set of these points is generally called the \textit{branch
locus} and is denoted $\mathcal{B}_{g}$.

With the exception of some low genus cases, $\mathcal{B}_{g}$ is disconnected
 \cite{BCI1} and displays all sorts of phenomena including having isolated points
\cite{K}. The points of $\mathcal{B}_{g}$ can be organized in strata
corresponding to surfaces with the same isometry group and the same
topological action of the isometry groups on the surfaces \cite{B} . The
closure of each strata is an equisymmetric set. Each equisymmetric set is
connected \cite{Na}, theorem 6.1 (see also \cite{MS}) and consists of surfaces with isometry group
containing a given finite group with a fixed topological action.

Furthermore, the set of Riemann surfaces that are Galois coverings of the Riemann sphere with a
fixed Galois group is a union of equisymmetric sets and corresponds to the loci of algebraic curves that admit a determined algebraic
expression. The most simple case is the (cyclic) $p$-gonal locus which is the
set of points in $\mathcal{M}_{g}$ corresponding to Riemann surfaces that are
$p$-fold cyclic coverings of the Riemann sphere. These are called cyclic $p$-gonal Riemann surfaces and in the case where $p$ is prime, having a $p$-gonal Galois
covering is equivalent to being cyclic $p$-gonal. From the algebraic
curve viewpoint these surfaces are those corresponding to an equation of type
\[
y^{p}=Q(x)
\]
where $Q$ is a polynomial. Note that if the $p$-gonal locus is connected it is
possible to continuously deform any cyclic $p$-gonal curve to other one keeping the $p$-gonality along the deformation (see \cite{SeSo}). 

The first result in this direction was to show the connectivity of the hyperelliptic locus \cite{Na} , i.e., the set of surfaces invariant
by a conformal involution with quotient a sphere. Note that as a subset of
Teichm\"{u}ller space it is disconnected. From the algebraic curve viewpoint
these surfaces are those corresponding to an equation of type
\[
y^{2}=Q(x)
\]
where $Q$ is a polynomial. In contrast, cyclic trigonal surfaces, i.e., surfaces with
an automorphism of order $3$ whose quotient is a sphere, are known to form
disconnected loci of $\mathcal{B}_{g}$ (see \cite{BSS}). These surfaces this
time correspond to surfaces with an equation of type
\[
y^{3}=Q(x).
\]
As $p$ increases, the behaviors of the corresponding $p-$gonal loci become more
exotic. For instance one can even find one (complex) dimensional equisymmetric
sets consisting of $p$-gonal surfaces and completely isolated inside
$\mathcal{B}_{g}$, the first example of this being equisymmetic sets of $11$-gonal
surfaces in genus $10$ \cite{CI2}. More generally, the connectivity of the branch locus for different types of group actions and properties of equisymmetric sets have been well studied and we refer the interested reader to \cite{BCIP}, \cite{BCI1},\cite{BCI2},\cite{BCI3},\cite{BI}, \cite{BSS}, \cite{CI3}, \cite{Se}.

In this paper we turn our attention to the Deligne-Mumford compacitification $\overline
{\mathcal{M}_{g}}$ of moduli space which is obtained by adding so-called
nodal surfaces. From the hyperbolic viewpoint, these are surfaces where a
geodesic multicurve has been pinched to length $0$. From the viewpoint of
algebraic curves, deformations correspond to variation in the coefficients or
roots of the polynomials and these nodal surfaces correspond to algebraic
curves with singularities.

Our first point of focus is on the connectivity of the cyclic trigonal locus
described above but this time in $\overline{\mathcal{M}_{g}}$. Our main result
is the following.

\begin{theorem}\label{thm:main1}
For $g\geq5$, there is an explicit nodal surface that lies in completion of
all the equisymmetric sets in the $3$-gonal locus. In particular, the set
of (cyclic) trigonal surfaces is connected in $\overline{\mathcal{M}_{g}}$.
\end{theorem}

Because there is a single surface that connects all of the equisymmetric sets,
from the equation viewpoint this implies that any two cyclic trigonal
equations can be deformed continuously from one to the other, where one allows
the passage through at most one nodal surface.

In light of the above one might expect that this type of phenomena continues
to occur for higher order cyclic $p$-gonal surfaces but in fact this fails in
general. To show this we restrict our attention to $1$-dimensional cyclic
$p$-gonal strata in genus $g=p-1$ for $p$ prime. The connectivity already
fails for $4-$gonal locus but this is less telling as the non-primality of $4$
induces different phenomena.

We show that already for $p=5$, one connected component of the cyclic
$5$-gonal locus in genus $4$ continues to be disconnected at the boundary.
We generalize this to higher genus to obtain the following.

\begin{theorem}
For $p\geq11$ prime, there are completely isolated one-dimensional strata in
$\overline{\mathcal{B}_{p-1}}$ corresponding to cyclic $p$-gonal surfaces.
\end{theorem}

Note that the techniques we use are not specific to $p\geq11$ but for lower
genus there aren't any isolated one dimensional strata in $\mathcal{M}_{g}$
let alone its completion.

\textbf{Organization.}

The article is organized as follows. We begin with a section of preliminaries
which includes well known results and certain basic lemmas we will need in the
sequel. We then prove the connectivity of cyclic trigonal surfaces in the
compactification. The last section deals with cyclic $p-$gonal surfaces in
genus $p-1$.

\textbf{Acknowledgements.}

The authors are grateful to Jeff Achter for pointing out the relationship between Theorem \ref{thm:main1} and \cite[Prop. 2.11]{AP}. The first and third authors are grateful to the University of Link\"{o}ping for hosting them for a stay during which substantial progress was made on this work.

\section{Preliminaries}

Let $f: S \to \hat{\C}$ be an $l$-fold branched covering and let $\{b_1,\hdots,b_r\}$ be the set of branched points. Let $o$ be a point in $\hat{\C}\setminus \{b_1,\hdots,b_r\}$ and assume $f^{-1}(o)= \{o_1,\hdots,o_l\}$.

We define the {\it monodromy} $\omega_f$ of $f$ as a map
$$
\omega_f: \pi_1(\hat{\C}\setminus \{b_1,\hdots,b_r\}, o) \to \Sigma_l = {\mathcal P}\{1,\hdots,l\} 
$$
as follows. Let $y=[\nu] \in  \pi_1(\hat{\C}\setminus \{b_1,\hdots,b_r\},o)$, where $\nu$ is a loop based in $o$. Now
$\omega_f(y)$ is a permutation on $\{1,\hdots,l\}$ which takes $u\in \{1,\hdots,l\}$ to $v\in \{1,\hdots,l\}$ (i.e. $\omega_f(y)(u)=v$) if  the lift $\tilde{\nu}$ of $\nu$ with origin in $o_u$ finishes in $o_v$.

\begin{figure}[h]
\leavevmode \SetLabels
\L(.28*.845) $O_1$\\
\L(.68*.83) $O_l$\\
\L(.44*.83) $O_u$\\
\L(.52*.86) $O_v$\\
\L(.495*.41) $O$\\
\L(.44*.67) $\tilde{\upsilon}$\\
\L(.545*.24) $\upsilon$\\
\endSetLabels
\begin{center}
\AffixLabels{\centerline{\epsfig{file =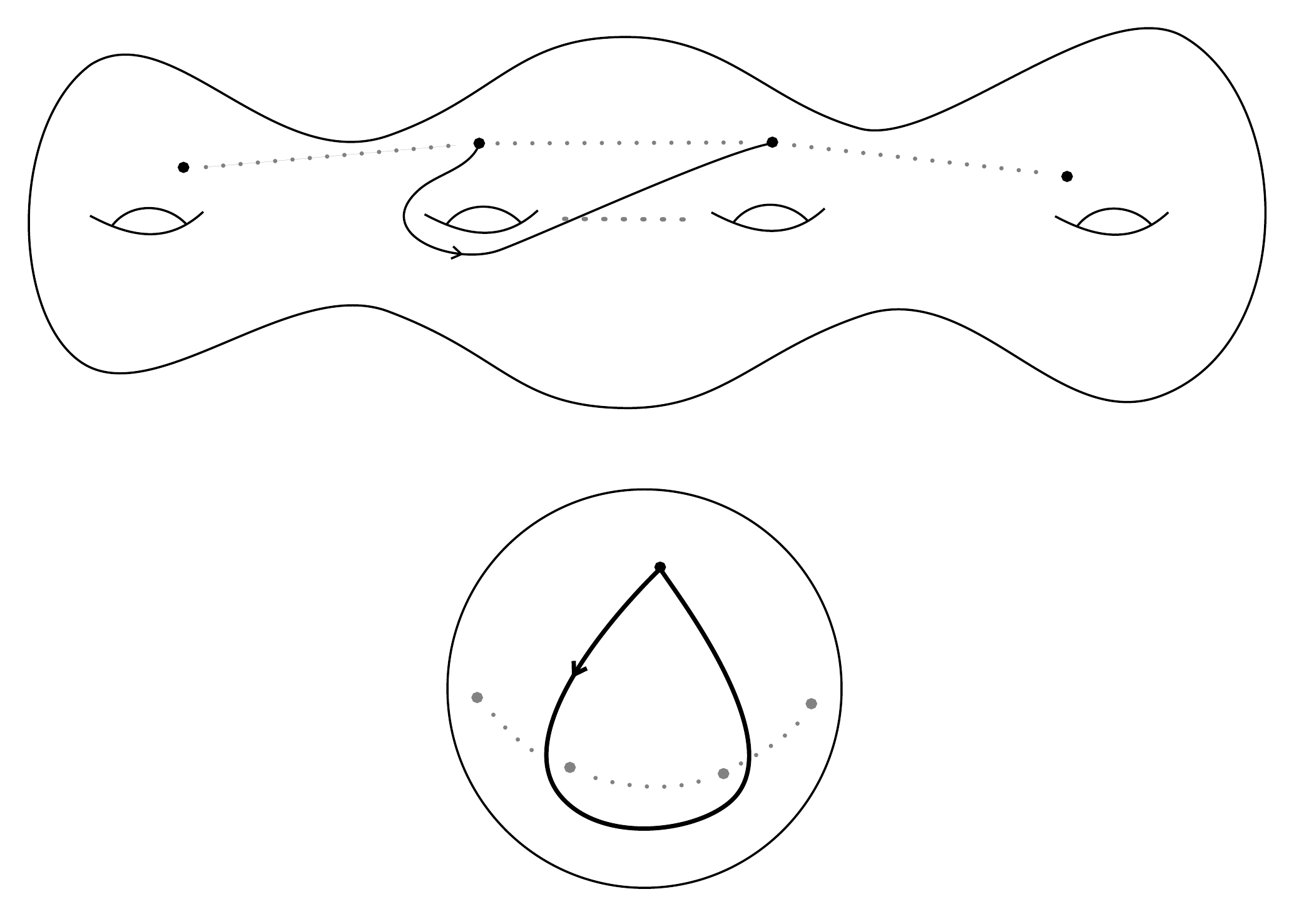,width=8.5cm,angle=0}}}
\vspace{-24pt}
\end{center}
\caption{An illustration of $\tilde{\upsilon}$ for a simple loop $\upsilon$} \label{fig:monodromy}
\end{figure}
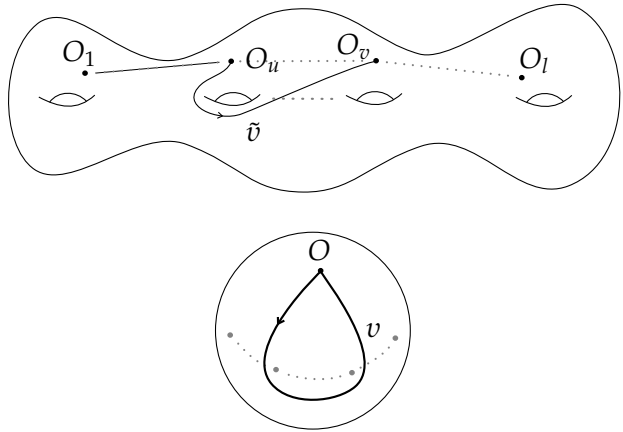

We call $x_i \in \pi_1(\hat{\C}\setminus \{b_1,\hdots,b_r\},o)$ a {\it meridian of a branch point} $b_i$ if it is represented by a simple loop $\xi_i$ based at $o$ that bounds a disk which contains $b_i$ and none of the other branch points. 

\begin{figure}[h]
\leavevmode \SetLabels
\L(.53*.92) $b_1$\\
\L(.54*.49) $b_i$\\
\L(.50*.08) $b_v$\\
\L(.64*.4) $\xi_i$\\
\endSetLabels
\begin{center}
\AffixLabels{\centerline{\epsfig{file =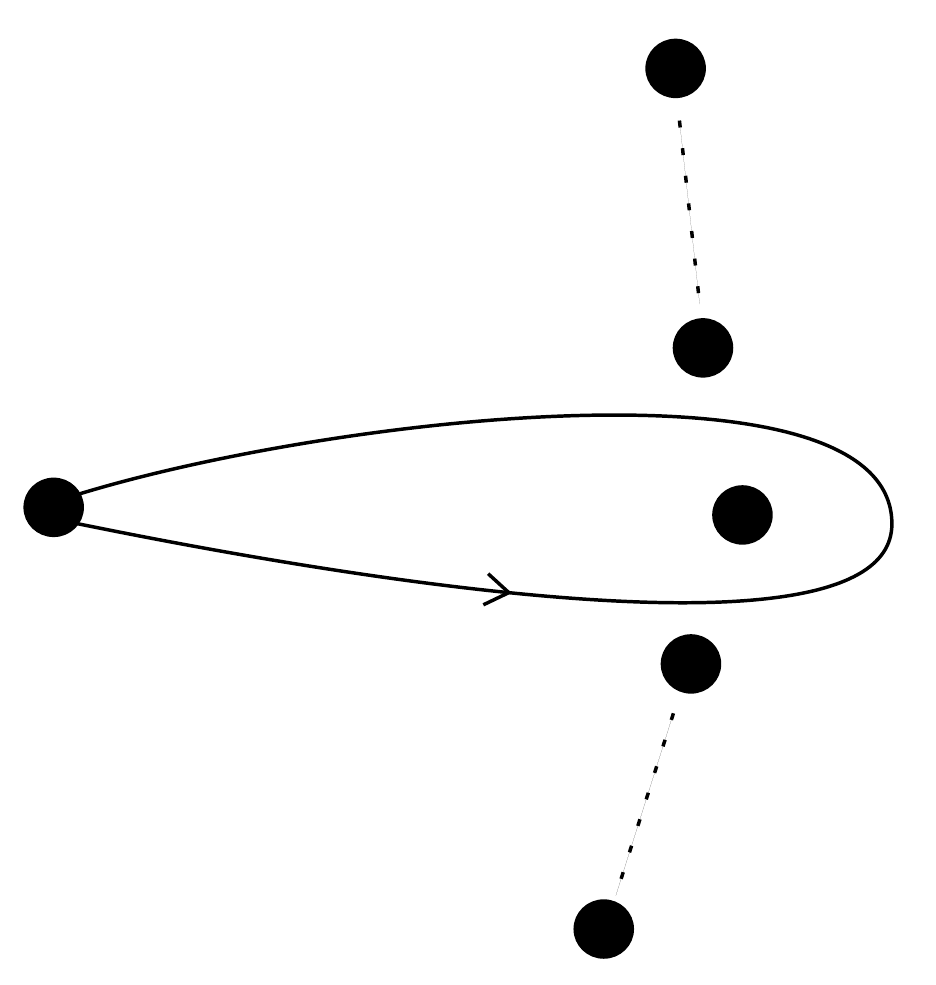,width=4.5cm,angle=0}}}
\vspace{-24pt}
\end{center}
\caption{The loop $\xi_i$ corresponding to the meridian $x_i$} \label{fig:meridian}
\end{figure}
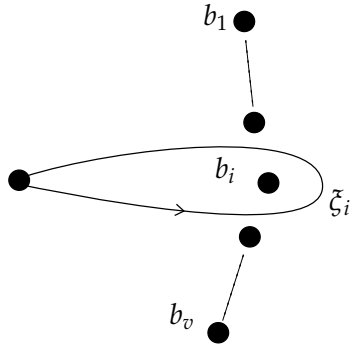

Meridians will be useful in the sequel as they are natural generators of the fundamental group $\pi_1(\hat{\C}\setminus \{b_1,\hdots,b_r\}, o)$ and monodromy representations are entirely determined by which permutation one associates to these elements. Specifically a set $x_1,\hdots,x_r$ is said to be a canonical set of generators of $\pi_1(\hat{\C}\setminus \{b_1,\hdots,b_r\},o)$ if the $x_i$ are all meridians and $\pi_1(\hat{\C}\setminus \{b_1,\hdots,b_r\},o)$ admits the following group presentation:
$$
<x_1,\hdots,x_r \mid x_1\hdots x_r =1 >.
$$

One fact we will use regularly is that there is a lower bound on the length of closed geodesics that pass through fixed points of automorphisms provided the automorphism has order $>2$. More specifically we have:
\begin{lemma}\label{lem:collar}
Let $S$ be a hyperbolic surface and $h$ an automorphism of $S$ of order $d>2$. Then there exists a constant $C_d>0$ such that every closed geodesic $\gamma$ that passes through at least one fixed point of $h$ satisfies
$$
\ell(\gamma)>C_d.
$$
Furthermore one can take $C_d\geq 2 \arccosh \frac{1}{\sin \frac{2\pi}{d}}$.
\end{lemma}
\begin{proof}
Consider $O:=S/<h>$ is an orbifold and $\gamma$ projects to a closed geodesic $\gamma'$ on $S/<h>$ which passes through an orbifold point of order $d$ and satisfies 
$$
\ell(\gamma)\geq \ell(\gamma').
$$
Now by the collar theorem on orbifold surfaces \cite{DP}, the length of any geodesic segment that passes through the collar is at least
$$
 \arccosh \frac{1}{\sin \frac{2\pi}{d}}
 $$
which gives the lower bound on the length of $\gamma$.
\end{proof}
A slightly more general lemma is indeed true but this is sufficient for our purposes.

\section{The topology of the branch locus of cyclic trigonal surfaces}

\begin{definition} A trigonal surface is a Riemann surface $S$ such that there exists a $3$-fold covering from $S$ to the Riemann sphere $\hat{\C}$. The morphism $f:S\to \hat{\C}$ is called the trigonal morphism. If the covering is regular, the surface is said to be cyclic trigonal. 
\end{definition}

One can express that a surface is $3$-gonal in terms of the monodromy \cite{ST}:

\begin{proposition} $f:S\to \hat{\C}$ is cyclic trigonal if and only if the monodromy representation is as follows:
\begin{eqnarray*}
\omega_f:  \pi_1(\hat{\C}\setminus \{b_1,\hdots,b_r\},o) & \to& \Sigma_3\\
x_i &\mapsto&(1,2,3) \mbox{ or } (1,3,2)
\end{eqnarray*}
\end{proposition}

We're interested in the case where are surfaces that can be endowed with a hyperbolic metric, thus we suppose that the genus $g$ of the surfaces is $\geq 2$. In this case, the surfaces admit a characterization in terms of Fuchsian groups.

\begin{proposition} $S$ is cyclic trigonal if and only if there exists a Fuchsian group $\Phi$ of signature $(0,\overset{r}{[3,\hdots,3]})$ and an epimorphism $\theta:\Phi \to C_3$ such that $S = \Hyp/\ker(\theta)$ and such that $\theta(x)\neq 1$ for each $x$ elliptic in $\Phi$.
\end{proposition}

Consider $x_1,\hdots,x_r$ a canonical set of generators of $\pi_1(\hat{\C}\setminus \{b_1,\hdots,b_r\},o)$, i.e. $\pi_1(\hat{\C}\setminus \{b_1,\hdots,b_r\},o)$ admits the following group presentation:
$$
<x_1,\hdots,x_r \mid x_1\hdots x_r =1 >.
$$
Now the permutation $t:=(1,2,3)$ generates $C_3$ as subgroup of $\Sigma_l = {\mathcal P}\{1,2,3\}$ and $\omega_f(x_i)=t$ or $t^{-1}$ for each $i$. We denote by $m_+$ the number of generators sent to $t$ and by $m_-$ the number of generators sent to $t^{-1}$ (by $\omega$). 

The quantities $m_+$ and $m_-$ satisfy the following equalities:
\begin{equation}\label{eqn:1}
m_++m_-= r = g+2
\end{equation}
and
\begin{equation}\label{eqn:2}
m_++ 2 m_- \equiv 0 \mod 3.
\end{equation}
The first equality is just by definition and the second comes from the equality in $C_3$ given by 
$$
\omega(x_1\hdots x_r)= \id_{C_3}
$$
from which it follows that
$$
t^{m_++2m_-}= \id_{C_3}.
$$

The following proposition due to Nielsen says that $m_+$ determines the morphisms topologically \cite{Ni}.
\begin{proposition}\label{prop:nielsen} Two cyclic trigonal morphisms $f_1: S_1 \to \hat{\C}$ and $f_2: S_2 \to \hat{\C}$ are topologically equivalent if and only if $m_+(S_1,f_1) = m_+(S_2,f_2)$.
\end{proposition}

Let $\M_g^{m_+=k}$ be the set of points in $\M_g$ corresponding to cyclic trigonal surfaces $(S,f)$ with topological type given by $m_+(S,f)=k$. 

\begin{proposition}[Consequence of \cite{Na} and \cite{G}]
For all $k$, $\M_g^{m_+=k}$ is connected and if $k\neq k'$
$$\M_g^{m_+=k} \cap \M_g^{m_+=k'} = \emptyset.$$
\end{proposition}

We can now state our first theorem.

\begin{theorem}\label{thm:cyclictrigonal} Let $I=\{ k \mid \M_g^{m_+=k}\neq \emptyset\}$. Then
$$
\displaystyle \bigcap_{k\in I} \overline{\M_g}^{\,m_+=k} \neq \emptyset.
$$
\end{theorem}

Note that this theorem can be also be deduced using computations in \cite[Prop. 2.11]{AP}, although here we shall give a complete proof in terms of hyperbolic structures on Riemann surfaces.
  
To prove the theorem we shall show the existence of an explicit nodal surface that belongs to the completion of each of the strata. To construct this surface we need the following lemma which guarantees the unicity of specific punctured surfaces which will serve as building blocks of our nodal surface.

\begin{lemma}\label{lem:unique}
Up to isometry, there are unique (and distinct) hyperbolic complete finite area surfaces that satisfy the following properties:
\begin{enumerate}
\item\label{case1} $Q$ is a once punctured torus with an isometry of order $3$,
\item $\alpha$ is a twice punctured torus with an automorphism of order $3$ with $2$ fixed points,
\item $X$ is a $4$ times punctured sphere with one automorphism of order $3$ with one fixed point.
\end{enumerate}
\end{lemma}
\begin{proof}
Case \ref{case1} is just the well known fact that the so-called modular torus is the unique punctured torus with an automorphism of order $3$. The conformal automorphism of order $3$ has $3$ fixed points, just choose one of them to remove and obtain a cusp (note there are isometries in that torus interchanging the three fixed points).

The surface $\alpha$ is obtained by considering the unique hyperbolic torus with two cusps in the conformal class of the modular torus with two of the fixed points removed and which become cusps. 

Now consider a pair of pants with three cusps as boundary. It clearly has an automorphism of order three which rotates the cusps and has $2$ fixed points. $X$ is the unique hyperbolic surface one obtains by removing one of the fixed points to obtain a cusp. Uniqueness is again guaranteed by the uniqueness of the conformal class of a thrice punctured pair of pants.
\end{proof}

\begin{proof}[Proof of Theorem \ref{thm:cyclictrigonal}]
We begin by endowing $S$ with a hyperbolic metric. Then $S$ admits an automorphism $h$ of order $3$ such that $S/<h>$ is a hyperbolic orbifold of genus $0$.

The construction of our nodal surface will be algorithmic and we begin by considering a specific pants decomposition of $S/<h>$ where the boundaries of pants are either simple closed geodesics or branched points.


Let $B$ be the set of branched points $\{b_1,\hdots,b_r\}$. Let $\omega_f:  \pi_1(\hat{\C}\setminus \{b_1,\hdots,b_r\},o)  \to \Sigma_3$ be the monodromy of $f:S\to S/<h> = \hat{\C}$.

Now $B=B_+\cup B_-$ where $B_+$ (resp. $B_-$) is the subset of $B$ consisting of points $b_i$ such that $\omega_f(x_i)=t$ (resp. $\omega_f(x_i)=t^{-1}$). 

We are going to arrange a maximum number of points of $B$ into triples where each triple lies in a $B_i$. Specifically let $T:=\lfloor{\frac{\#B_+}{3}}\rfloor +\lfloor{\frac{\#B_-}{3}}\rfloor$. And set 
$\{b^k_1,b^k_2,b^k_3\}_{k=1}^{T}$ so that 
$$
\{b^k_1,b^k_2,b^k_3\} \subset B_i
$$
and each $b\in B$ lies at most one triple. 

\underline{Initial step}

We now construct a first pair of pants. Consider disjoint simple closed geodesics $\gamma^1_1,\gamma^1_2$ such that $\gamma^1_1,b^1_1,b^1_2$ and $\gamma^1_1,\gamma^1_2, b^1_3$ are the boundaries of embedded pairs of pants. (There are infinitely many choices for such curves.)

\begin{figure}[h]
\leavevmode \SetLabels
\L(.395*.47) $b_1^1$\\
\L(.507*.47) $b_2^1$\\
\L(.63*.47) $b_3^1$\\
\L(.50*.1) $\gamma_1^1$\\
\L(.66*.1) $\gamma_2^1$\\
\endSetLabels
\begin{center}
\AffixLabels{\centerline{\epsfig{file =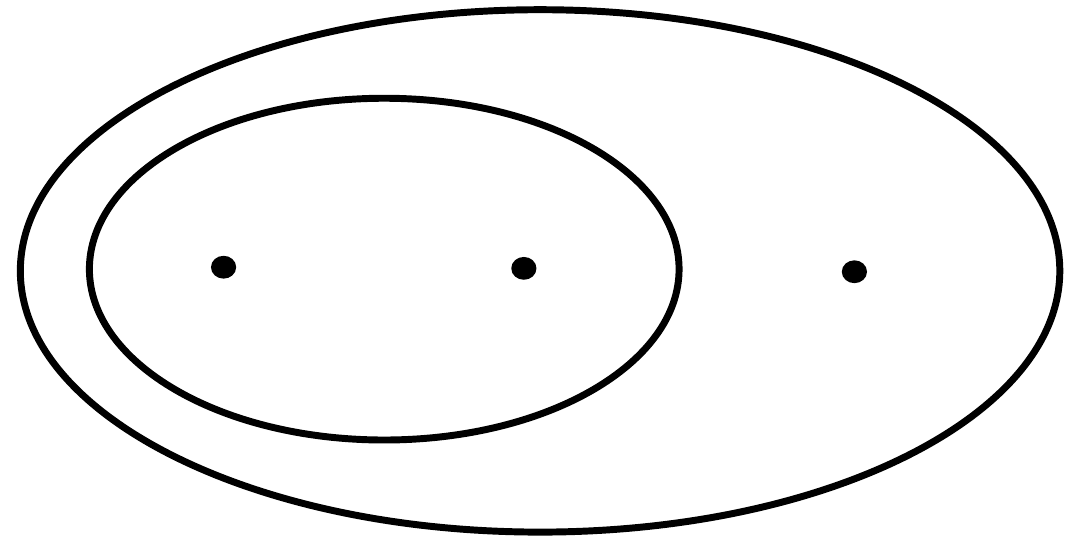,width=6cm,angle=0}}}
\vspace{-24pt}
\end{center}
\caption{The curves $\gamma_1^1$ and $\gamma_1^2$} \label{fig:step1}
\end{figure}
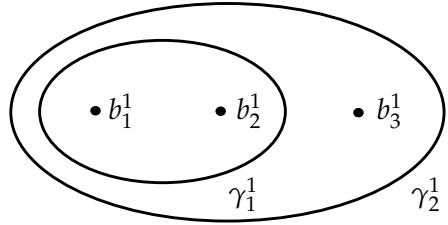

\underline{General step}

Now for $k\in \{2,\hdots,T\}$ we construct simple closed geodesics $\gamma_1^k,\gamma_2^k,\gamma_3^k$, belonging to the portion of $S/<h>$ from which we have removed the pants constructed previously, with the following properties:
$$\gamma^{k-1}_3, \gamma^k_1, b_1^k ;$$
$$\gamma^{k}_1, \gamma^k_2, b_2^k$$
and
$$\gamma^{k}_2, \gamma^k_3, b_3^k$$
are the boundary curves of embedded pairs of pants. (Again there are many choices for these curves.)

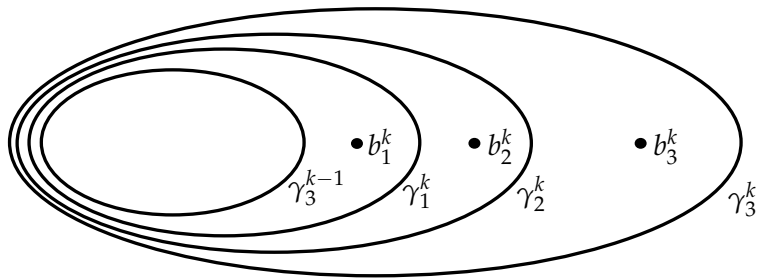
\begin{figure}[h]
\leavevmode \SetLabels
\L(.49*.46) $b_1^k$\\
\L(.595*.46) $b_2^k$\\
\L(.74*.46) $b_3^k$\\
\L(.42*.31) $\gamma_3^{k-1}$\\
\L(.52*.30) $\gamma_1^k$\\
\L(.62*.29) $\gamma_2^k$\\
\L(.805*.27) $\gamma_3^k$\\
\endSetLabels
\begin{center}
\AffixLabels{\centerline{\epsfig{file =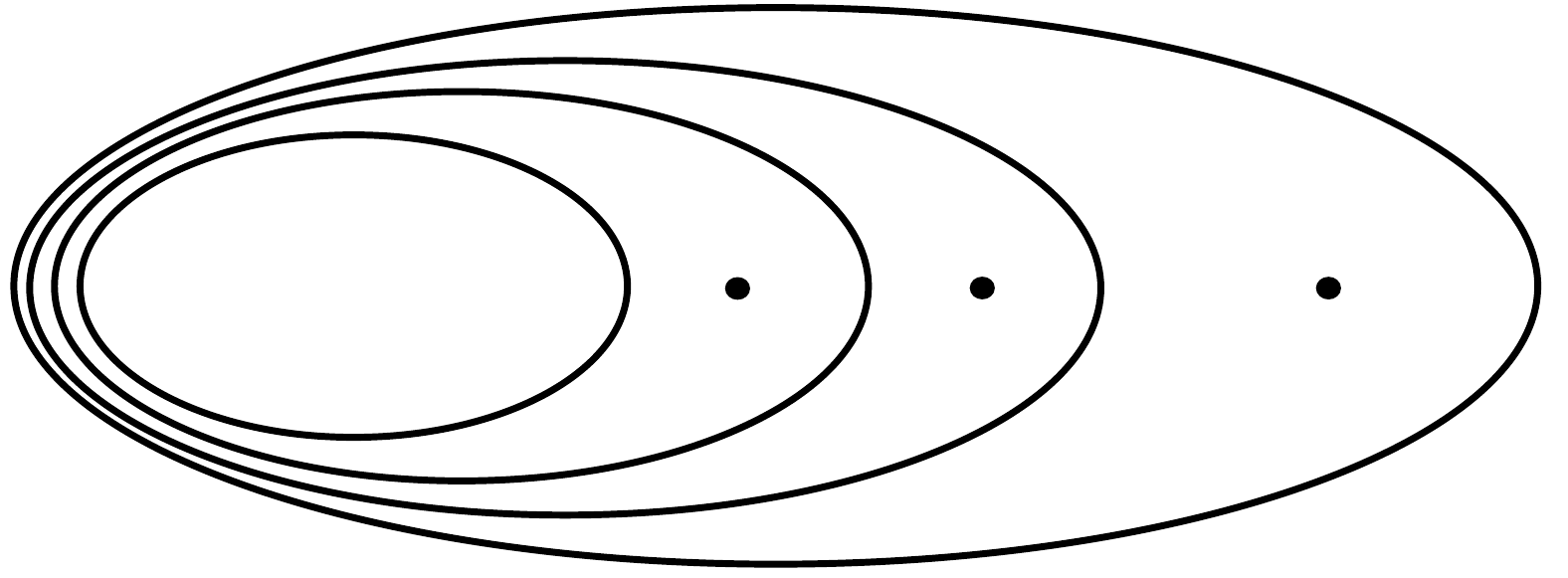,width=10cm,angle=0}}}
\vspace{-24pt}
\end{center}
\caption{The general step} \label{fig:stepk}
\end{figure}

Recall that $m_+$ and $m_-$ satisfy equations \ref{eqn:1} and \ref{eqn:2}. As such there are 3 cases to consider depending on the number of branch points that do not belong to the pants we have constructed. Following equation \ref{eqn:2}, we have 
$$
m_+\equiv m_{-} \mod 3
$$
and thus the number of points is either $0,2$ or $4$, and the number of remaining points from $B_-$ is the same as the number from $B_+$.

- If there are none, then note that the final 2 curves of our pants decomposition were in fact trivial.

- Suppose we have $1$ in both $B_+$ and $B_-$ (say $b$ and $\tilde{b}$: then the set of curves we have constructed form a full pants decomposition of $S/<h>$ and the final pair of pants is $\gamma_3^T, b, \tilde{b}$.

- Suppose we have $2$ in both $B_+$ and $B_-$ (say $b_1,b_2\in B_+$, $\tilde{b}_1,\tilde{b}_2\in B_-$). Consider disjoint curves $\gamma,\tilde{\gamma}$ such that $\gamma_3^T,\gamma,b_1$ and $\tilde{\gamma}, \tilde{b}_1, \tilde{b}_2$ form pants. These curves form a final pair of pants $\gamma, \tilde{\gamma}, b_2$.

We are now ready to construct our nodal surface. We claim that by pinching the curves of our pants decomposition on $S/<h>$ and lifting via $h$ we obtain a unique surface, independently of the monodromies of the points $b_1,\hdots,b_r$ we began with. 

We begin by lifting the first pair of pants: by construction the two branched points have the same monodromy and as such, the curve $\gamma_1^1$ lifts to a unique simple closed geodesic invariant by the isometry. Now by Riemman-Hurwitz, this implies that the lift of the pair of pants is a one holed torus with an isometry of order $3$. When we pinch $\gamma_1^1$ to length $0$, in light of lemma \ref{lem:unique}, the one holed torus becomes the modular torus $Q$.

Now we lift the pair of pants with boundaries $\gamma_1^1, b_3^1, \gamma_2^1$. The points $b_1^1, b_2^1, b_3^1$ all the have the same monodromy. If we think of $\gamma_2^1$ as being an element of $\pi_1(\hat{\C}\setminus \{b_1,\hdots,b_r\},o)$ (by giving it an orientation and chosing the point $o$ suitably) then it in light of this, it would be sent by $\omega_f$ to $\id_{\Sigma_3}$. As such it lifts to three distinct curves on $S$. The lift of the pair of pants with boundaries $\gamma_1^1,b_3^1, \gamma_2^1$ has $4$ boundary curves on $S$ and by Riemann-Hurwitz is has genus $0$. As such we obtain a four holed sphere with an autorphism of order $3$ with $1$ fixed point. Again by pinching the curves, this subsurface lifts to $X$. 

Now we argue similarly for each subsequent sequence of three pairs of pants. After pinching these lift to $X$, $\alpha$ and $X$. (The $\alpha$ appears because via Riemann-Hurwitz and the monodromy, the second pair of pants lifts to a two holed torus with an automorphism of order $3$ and two fixed points. In light of lemma \ref{lem:unique}, by pinching the subsurface is $\alpha$.)

As such the lift of our surface is 
$$
Q+X+(X+\alpha+X)+\cdots...
$$
where we somewhat loosely denote by $+$ a type of ``stable" sum of surfaces as in the following figure.

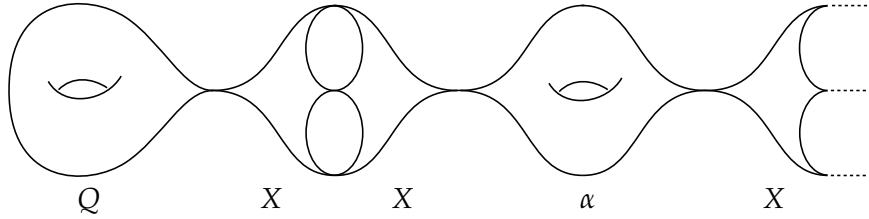
\begin{figure}[h]
\leavevmode \SetLabels
\L(.17*-.15) $Q$\\
\L(.33*-.15) $X$\\
\L(.445*-.15) $X$\\
\L(.61*-.15) $\alpha$\\
\L(.77*-.15) $X$\\
\endSetLabels
\begin{center}
\AffixLabels{\centerline{\epsfig{file =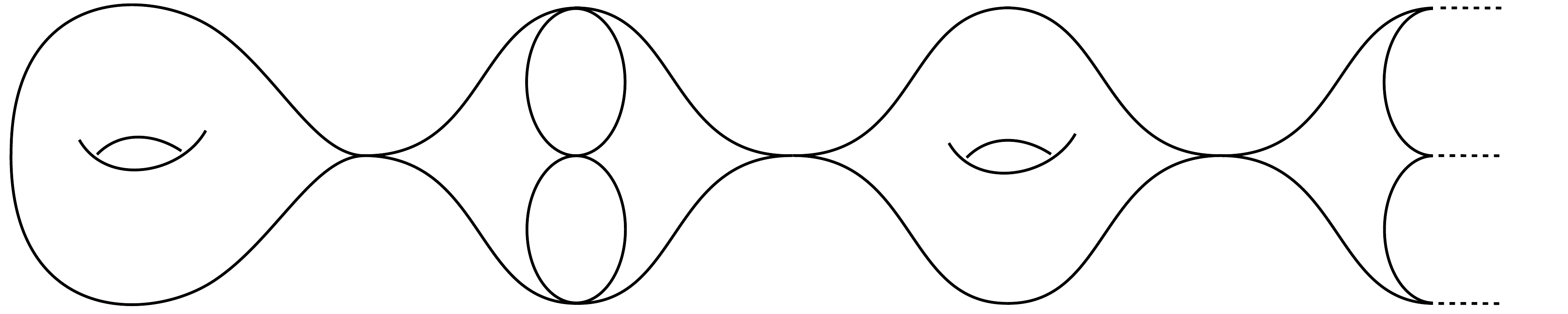,width=12cm,angle=0}}}
\vspace{-14pt}
\end{center}
\caption{The beginning of the lift} \label{fig:beginning}
\end{figure}

There are three cases to consider, depending uniquely on properties of the number $g$ (or equivalently $r$). 

\underline{Case 1: $r \equiv 0 \mod 3$}

In this case, after obtaining in the lift $Q+X$ we have lifted $T-2$ copies of $X+\alpha+X$. There is now a final triple of branch points with the same monodromy and the curve $\gamma_1^1$ surrounds two of the branch points to form a pair of pants. As in the initial lift, this pair of pants lifts to $Q$. The surface then ends with $X+Q$. The final surface is (in our loose notation)
$$
Q+X+\sum_{T-2}( X+\alpha+X ) + X+Q.
$$

\underline{Case 2: $r \equiv 1 \mod 3$}

Arguing similarly we obtain in this case
$$
Q+X+\sum_{T-1}( X+\alpha+X ) + X+\alpha+Q.
$$

\underline{Case 3: $r \equiv 2 \mod 3$}

Arguing similarly we obtain in this case
$$
Q+X+\sum_{T-1}( X+\alpha+X ) + Y
$$
where by $Y$ we mean the unique hyperbolic thrice punctured sphere. These cases are illustrated in figure \ref{fig:fullsurface}.

\begin{figure}[h]
\leavevmode \SetLabels
\L(.82*.82) $X+Q$\\
\L(.93*.49) $X+\alpha+Q$\\
\L(.72*.16) $Y$\\
\endSetLabels
\begin{center}
\AffixLabels{\centerline{\epsfig{file =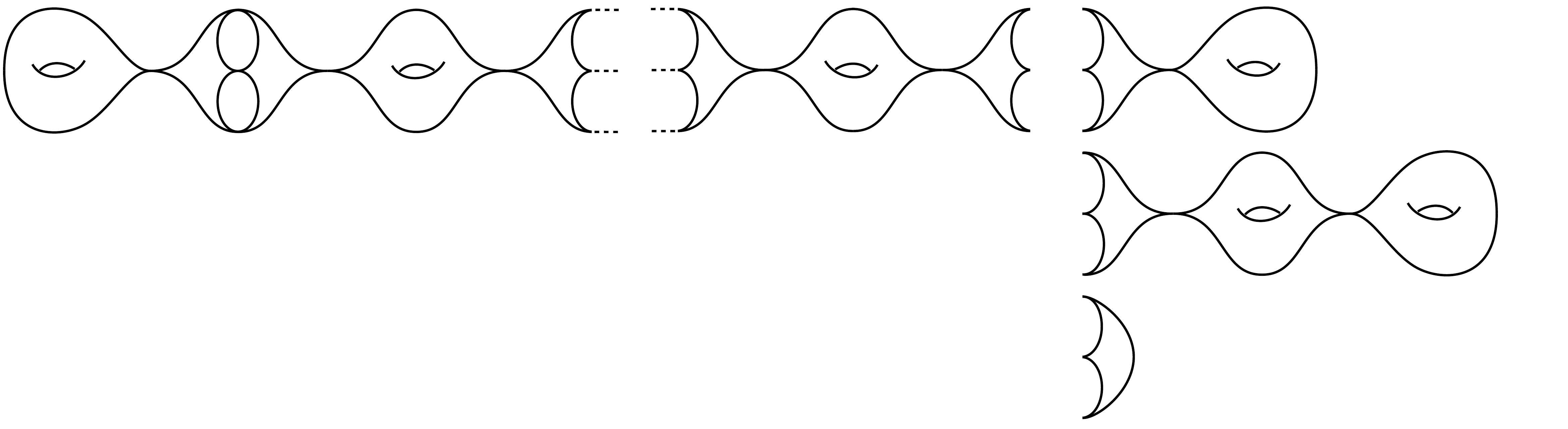,width=14cm,angle=0}}}
\vspace{-24pt}
\end{center}
\caption{The full surface with the three possible end cases} \label{fig:fullsurface}
\end{figure}

\end{proof}

\section{One dimensional cyclic $p$-gonal loci for prime $p>3$}


In light of the above, the incurable optimist might believe that the set of cyclic $p$-gonal surfaces for $p$ prime is always connected in the completion of moduli space. In this section we show that this fails in general. 

Recall that $f:S\to \hat{\C}$ is a cyclic $p$-gonal covering if it is regular cyclic covering. We recall the well-known characterization of such coverings in terms of monodromy.

\begin{proposition} The covering $f\to \hat{\C}$ is a cyclic $p$-gonal covering if and only if there is a canonical set of generators $\{x_1,\hdots,x_r\}$ with monodromy representation as follows:
\begin{eqnarray*}
\omega_f:  \pi_1(\hat{\C}\setminus \{b_1,\hdots,b_r\},o) & \to& \Sigma_p\\
x_i &\mapsto&(1,2,\hdots,p)^{j_i},\,\,j_i\in\{1,2,\hdots,p-1\}.
\end{eqnarray*}
\end{proposition}
In this case note that we have
$$
r = \frac{2g}{p-1} +2.
$$
As the product of the generators is the identity we have:
$$
\sum_{i=1}^r j_i \equiv 0 \mod p.
$$

We now pass to a first example that will serve as a guide for what follows.

\subsection{Cyclic $5$-gonal surfaces in genus $4$}

Let $g=4$. Following Nielsen, there are three topological types of cyclic $5$-gonal coverings of genus $4$ surfaces. There are given by the monodromy types
$$\omega_f: \pi_1(\hat{\C}\setminus \{b_1,\hdots,b_4\},o)\to <t> \subset \Sigma_5$$
with $t:=(1,2,3,4,5)$ and with the property that the order of $\omega_f(x_i)$ is $5$. These are given by
\begin{enumerate}
\item $\omega_f(x_1)=\omega_f(x_2)=\omega_f(x_3)=t$ and $ \omega_f(x_4)=t^2$,
\item $\omega_f(x_1)=\omega_f(x_2)=t$ and $\omega_f(x_3)=\omega_f(x_4)=t^4$,
\item $\omega_f(x_1)=t$, $\omega_f(x_2)=t^2$, $\omega_f(x_3)=t^3$ and $\omega_f(x_4)=t^4$.
\end{enumerate}
Note that this means that there are three topological types of cyclic  $5$-gonal surfaces, each given by the monodromies specified above. It is a well known fact that they live in distinct connected components of the (cyclic) $5$-gonal locus in $\M_4$ (see \cite{CI1}). Specifically $i=1,2,3$ each
$$
\M_4^{i} := \{ S\in \M_4 \mid \mbox{ there exists a cyclic $5$-gonal morphism $f:S\to \hat{\C}$  of type i }\}$$
is connected with $\dim_\C \M_4^i=1$ and that for $i,j=1,2,3$
$$
\M_4^{i}\cap \M_4^j = \emptyset
$$
if $i\neq j$. 

Our observation in this setup is the following:
\begin{theorem}\label{thm;genus4}
$$
\overline{\M}_4^1 \cap \overline{\M}_4^j = \emptyset
$$
for $j\neq 1$ and
$$
\overline{\M}_4^2 \cap \overline{\M}_4^3 \neq \emptyset
$$
\end{theorem}

\begin{proof}
In this proof, we assume that our surfaces are endowed with their unique hyperbolic metrics. In particular the cyclic $5$-gonal covering becomes an automorphism of the hyperbolic metrics.

Assume that $S \in \overline{\M}_4^j \setminus \M_4$. Then there exists for $k=1,j$ sequences $\{S^{(k)}_i\in \M_4\} \to S$. Denote $\mu^{(k)}_i\subset S_i$ the multicurve whose length approaches $0$ as $i\to \infty$. 

If we denote $a_i^{(k)}$ the cyclic $5$-gonal automorphism of $S_i^{(k)}$, the multicurves $\mu_i^{(k)}$ must be $a_i^{(k)}$-invariant. In particular, it is important to observe that $\mu_i^{(k)}$ is the lift of a simple closed geodesic on $S_i^{(k)}/<a_i^{(k)}>$. This follows from the fact that $\mu_i^{(k)}$ descends on $S_i^{(k)}/<a_i^{(k)}>$ to a connected curve whose length must also go to $0$ and by Lemma \ref{lem:collar}, this curve cannot pass through the fixed points of order $5$.

There are three different topological types of simple closed geodesics on an orbifold of genus $0$ with $4$ points of order $5$ (the quotients $S_i^{(k)}/<a_i^{(k)}>$ are all of this type). Observe that there is a unique genus $0$ hyperbolic orbifold with one cusp and two orbifold points of order $5$. As such, there are at most three possible stable surfaces (up to isometry) in each $\overline{\M}_4^{(k)} \setminus \M_4$. We will describe these surfaces for each stratum.

Each simple closed geodesic on a $S_i^{(k)}/<a_i^{(k)}>$ surrounds two orbifold points on one side and two on the other. Once we pinch, we obtain pants with two orbifold points and a cusp. We are interested in the isometry types of surfaces that these pants lift to. Their isometry types are clearly determined by the monodromies but also clearly different monodromies can lift to isometric pieces. For instance, if the two monodromies are $t,t$ or $t^{-1},t^{-1}$, then the pants lift to isometric pieces. We claim that the pants can lift to exactly three isometrically distinct surfaces:

\underline{Case 1: $\{t,t\}, \{t^{-1},t^{-1}\}$}

In this case we get a genus $2$ surface with one cusp with an automorphism of order $5$ with two fixed points with the same rotation index. Forgetting the cusp, this is a conformally the unique genus $2$ surface with a conformal automorphism of order $5$. We denote this surface by $P_1$.

\underline{Case 2: $\{t,t^2\},\{t,t^3\},\{t^2,t^4\},\{t^3,t^4\} $}

In this case we get a genus $2$ surface with one cusp with an automorphism of order $5$ with two fixed points but with the different rotation indices. Note that although this surface is conformally equivalent to the one in the previous case when one forgets the cusp, they are not isometric as the topological types of the coverings are different, and a result by \cite{G}  tells us that there is only one topological type of cyclic coverings from a given surface of genus $2$ to the sphere. We denote this surface by $P_2$.

\underline{Case 3: $\{t,t^{-1}\}=\{t,t^4\}, \{t^2,t^3\} $}

In this case, we obtain a sphere with $5$ cusps with an automorphism of order $5$ with two fixed points and which permutes the cusps. We denote this surface by $P_3$.

Let us pause for a moment to consider the geometries of $P_1$, $P_2$ and $P_3$. They can be constructed as follows. Consider the unique regular hyperbolic ideal pentagon. (In figures \ref{fig:5gons} we have schematically drawn this pentagon as Euclidean.) Now paste two copies along a common edge ``without" shearing, i.e, such that the geodesic between the two centers of the pentagons meets the common edge at a right angle. This gives an octogon. All three surfaces can now be obtained by pasting the octogon in different ways as illustrated in the figures. The fact that these are indeed the correct surfaces follows from the fact that they have the appropriate isometry groups which descend to the pants with the appropriate monodromies and by the uniqueness arguments outlined above.

\begin{figure}[h]
\leavevmode \SetLabels
\L(.00*.38) $a_1$\\
\L(.00*.62) $a_1$\\
\L(.04*.09) $a_2$\\
\L(.04*.90) $a_2$\\
\L(.205*.09) $a_3$\\
\L(.205*.90) $a_3$\\
\L(.25*.38) $a_4$\\
\L(.25*.62) $a_4$\\
\L(.135*.253) ${\,}_1$\\
\L(.125*.685) ${\,}_1$\\
\L(.367*.38) $a_1$\\
\L(.367*.62) $a_2$\\
\L(.407*.09) $a_2$\\
\L(.407*.90) $a_4$\\
\L(.572*.09) $a_3$\\
\L(.572*.90) $a_1$\\
\L(.617*.38) $a_4$\\
\L(.617*.62) $a_3$\\
\L(.502*.253) ${\,}_2$\\
\L(.492*.685) ${\,}_1$\\
\L(.734*.38) $a_1$\\
\L(.734*.62) $a_4$\\
\L(.774*.09) $a_2$\\
\L(.774*.90) $a_3$\\
\L(.939*.09) $a_3$\\
\L(.939*.90) $a_2$\\
\L(.984*.38) $a_4$\\
\L(.984*.62) $a_1$\\
\L(.869*.253) ${\,}_4$\\
\L(.859*.685) ${\,}_1$\\
\endSetLabels
\begin{center}
\AffixLabels{\centerline{\epsfig{file =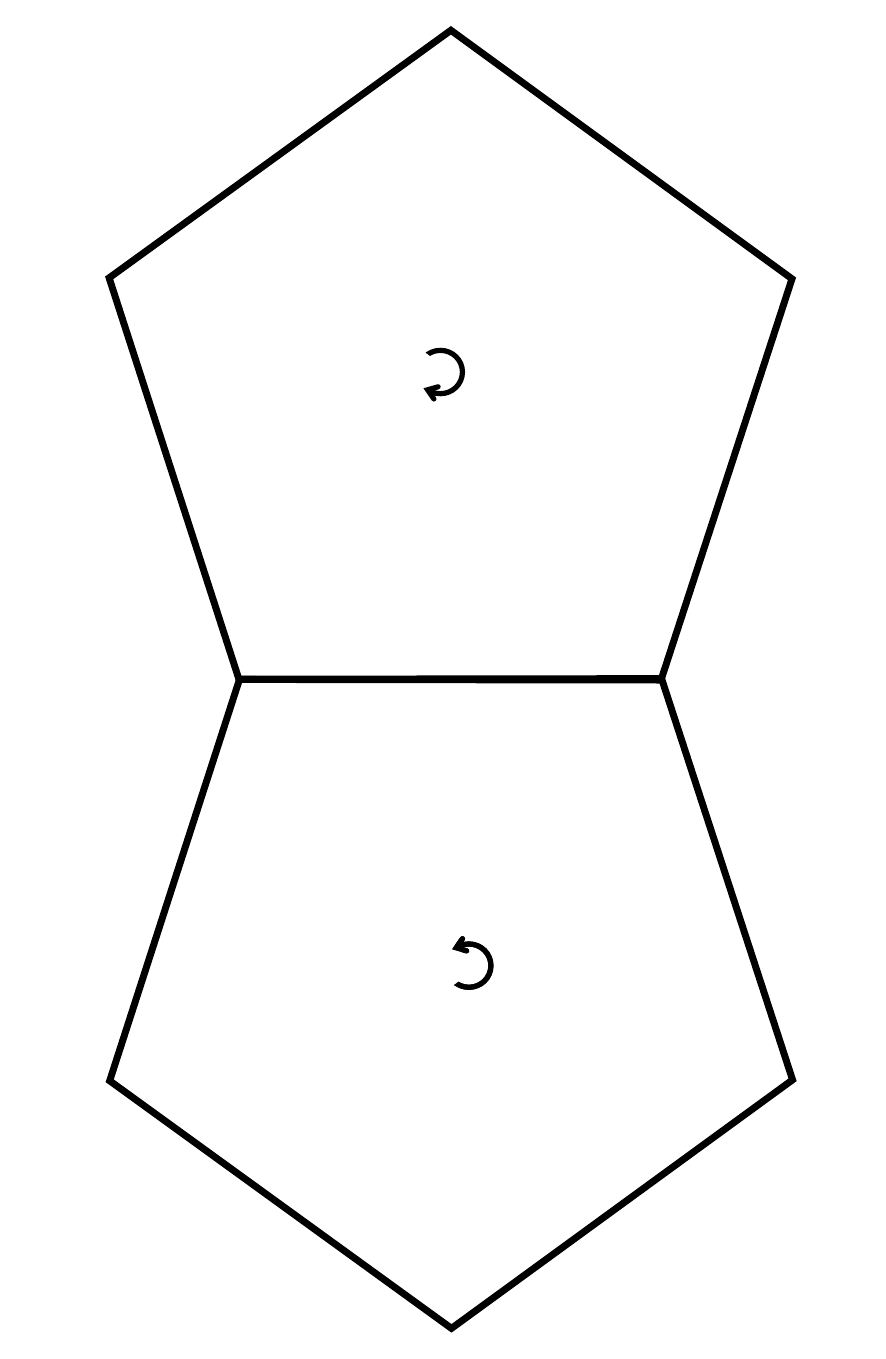,width=5.5cm,angle=0},\epsfig{file =5gons.pdf,width=5.5cm,angle=0},\epsfig{file =5gons.pdf,width=5.5cm,angle=0}}}
\vspace{-24pt}
\end{center}
\caption{The surfaces $P_3$ (sphere), $P_1$ (torus), $P_2$ (genus $2$)} \label{fig:5gons}
\end{figure}

We now look at which surfaces can lie on the boundary of the different strata. 

For $ \M_4^1 $ there is only one such surface as all three choices of simple closed geodesic surround two orbifold points whose monodromy is $t$. The other pair of pants has points with monodromy $\{t,t^2\}$. Thus the surface is given by a copy of $P_1$ and a copy of $P_2$ which are glued at their cusps. We denote the surface thus obtained somewhat loosely by $P_1 + P_2$.

For $ \M_4^2$ there are two distinct possibilities, i.e., the monodromies are given by the pairs $\{t,t\},\{t^{-1},t^{-1}\}$ or $\{t,t^{-1}\}, \{t,t^{-1}\}$. In the first case we obtain $P_1+P_1$ and in the other $P_3+ P_3$.

For $ \M_4^3$ we obtain also only two distinct possibilities, i.e., the monodromies are given by the pairs $\{t,t^2\},\{t^{3},t^{4}\}$ or $\{t,t^{3}\}, \{t^2,t^{4}\}$. This gives $P_2 + P_2$ and $P_3+P_3$.

By the above analysis it is clear that $ \M_4^2$ and $ \M_4^3$ meet at the boundary (at a unique point $P_3+P_3$) and $\M_4^1$ is disjoint from the other two.
\end{proof}

\subsection{$p$-gonal with $p$ prime and $g=p-1$}

We now consider a generalization of the above example. The proof follows the same outline is indeed almost identical. 

Let $g=p-1$. Again following Nielsen, we can classify topological types of cyclic $p$-gonal coverings of genus $p-1$ surfaces. There are given by the monodromy types
$$\omega_f: \pi_1(\hat{\C}\setminus \{b_1,b_2,b_3,b_4\},o)\to <t> \subset \Sigma_p$$
with $t:=(1,2,\hdots,p)$ and with the property that the order of $\omega_f(x_i)$ is $p$. For simplicity, we have indexed the $\omega_f$ by their type. These are given by 
\begin{enumerate}
\item $\omega_1(x_1)=\omega_1(x_2)=\omega_1(x_3)=t$ and $\omega_1(x_4)=t^{p-3}$,
\item $\omega_2(x_1)=\omega_2(x_2)=t$, and $\omega_2(x_3)=\omega_2(x_4)=t^{p-1}$,
\item $\omega_{3,i}(x_1)=t$, $\omega_{3,i}(x_2)=t^i$, $\omega_{3,i}(x_3)=t^{-i}$ and $\omega_{3,i}(x_4)=t^{p-1}$,
\item $\omega_{4,i}(x_1)=t$, $\omega_{4,i}(x_2)=t$, $\omega_{4,i}(x_3)=t^{i}$ and $\omega_{3,i}(x_4)=t^{p-2-i}$,
\item $\omega_{5,i,j}(x_1)=t$, $\omega_{5,i,j}(x_2)=t^i$, $\omega_{5,i,j}(x_3)=t^{j}$ and $\omega_{5,i,j}(x_4)=t^{p-1-i-j}$.
\end{enumerate}

Note that types 3, 4, and 5 contain several subtypes of monodromies. As before we consider the strata of moduli space corresponding to each type. Specifically we denote
$$
\overline{\M}_{p-1}^{\,I}:=S\in \M_{p-1} \mid \mbox{ there exists a  $p$-gonal morphism $f:S\to \hat{\C}$  of type I}\}
$$

We can now state our main theorem.
\begin{theorem}\label{thm:p-1}
$\overline{\M}_{p-1}^{\,(5,i,j)}$ is completely isolated in $\overline{\BB_g}$.
\end{theorem}

Observe that for the theorem to be true, $\M_{p-1}^{\,(5,i,j)}$ must be completely isolated in ${\BB_g}$ and this is true as was shown in \cite{CI2}.

\begin{proof}
Our first observation is that up to isometry, there are exactly $\frac{p+1}{2}$ hyperbolic surfaces with an automorphism of order $p$ with two fixed points and whose quotient is a sphere with a single cusp. To see this, we will generalize the case by case analysis of the lifts of the pairs of pants in the proof of Theorem \ref{thm;genus4}.

Consider $O_{p,p,\infty}$ be the (hyperbolic) orbifold of genus $0$ with two orbifold points of order $p$ and a cusp. The cyclic $p$-gonal coverings are given by the monodromies 
$$
\theta: \pi_1{\mathrm{Orb}}(O_{p,p,\infty}) = <y_1,y_2 \mid y_1^p=y_2^p>  \to  <t>\subset \Sigma_p$$
(where $\pi_1{\mathrm{Orb}}$ denotes the orbifold fundamental group) and we define a map 
$$
(y_1,y_2) \mapsto (\theta(y_1),\theta(y_2))= (t^a,t^b).
$$
By the result of Gonzalez-Diez \cite{G} , two such maps $(t^a,t^b)$ and $(t^{a'},t^{b'})$ induce equivalent surfaces if and only if there exists a $c$ such that
$$
a'\equiv ca\mod (p)
$$
and
$$
b'\equiv cb \mod (p) \mbox{ or } b' \equiv c b^{-1} \mod (p).
$$
We denote the monodromy types by $(i,j)$ if it is represented by $(t^i,t^j)$. Observe that in each equivalence class, there is a representative of type $(1,j)$.

We denote $P_j$ the covering of $O_{p,p,\infty}$ given by the monodromy of type $(1,j)$. 

We proceed as in the example and we now analyze the surfaces obtained at the limit in the different types discussed in the beginning of the section. 

\underline{Type 1:}
Here we obtain $P_1+P_{p-3}$ as we have the lift of a $O_{p,p,\infty}$ of type $(1,1)$ and one of type $(1,p-3)$. We proceed in the same way for each of the subsequent types.

\underline{Type 2:} $P_1+P_1$, $P_{p-1}+P_{p-1}$ 

\underline{Type 3:} $P_{p-1}+P_{p-1}$, $P_i+P_i$, $P_{-i}+P_{-i}$

where $2\leq i \leq \frac{p-1}{2}$

\underline{Type 4:} $P_1+P_{\frac{p-i-2}{i}}$, $P_{i}+P_{p-i-2}$

where $2\leq i \leq \frac{p-1}{2}$

\underline{Type 5:} $P_i+P_{-1-\frac{i+1}{j}}$, $P_j+P_{-1-\frac{j+1}{i}}$, 
$P_{\frac{j}{i}}+P_{p-1 - i -j}$

where $2\leq i \leq \frac{p-1}{2}$, $i<j\leq p-3$ with $p-1-i-j \not\in \{1,i,j,p-1,-i,-j\}$

Through the equivalences above, it is straightforward to check that the surfaces appearing in Type 5 do not appear in any of the other cases. It remains to show that they are distinct from each other.

Let us show that the surfaces of type $5$ are distinct for distinct equisymmetric sets in $\M_{p-1}$. We must prove that if either:

1. $P_i+P_{-1-\frac{i+1}{j}} = P_i'+P_{-1-\frac{i'+1}{j'}}$,

2. $P_i+P_{-1-\frac{i+1}{j}} = P_{\frac{j'}{i'}}+P_{p-1 - i' -j'}$ or

3. $P_{\frac{j}{i}}+P_{p-1 - i -j} = P_{\frac{j}{i}}+P_{p-1 - i -j}$ 

then $\M_{p-1}^{\,(5,i,j)} = \M_{p-1}^{\,(5,i',j')}$.

Let us assume that we are in the situation described in the point 1 (the other cases are similar). If $i = i'$ and $-1-\frac{i+1}{j} = -1-\frac{i'+1}{j'}$ then it is clear that $(i,j) = (i',j')$. If $i = -1-\frac{i'+1}{j'}$ and $i' = -1-\frac{i+1}{j}$ then using automorphisms of $\pi_1(\hat{\C}\setminus \{b_1,b_2,b_3,b_4\},o)$ and $C_p$ and elementar number theory it is easy to show that $\omega_{5,i,j}$ is equivalent to $\omega_{5,i',j'}$. 

This proves that the sets $\overline{\M}_{p-1}^{\,(5,i,j)}$ are isolated among the set of cyclic $p$-gonals. We now show that they cannot meet another type of surface from the branch locus on the boundary.

Suppose the contrary. The limit surface $S$ will have then $h$ the $p$-gonal automorphism and an automorphism of different type, i.e., non $p$-gonal. This second automorphism $h'$ will induce an automorphism on the quotient $S/h$ because $h$ is normal inside the automorphism group of the surface by \cite{G} . Now $h'$ descends to an autormorphism $\bar{a}:O_{p,p,\infty}$ such that 
$$
\bar{a}_* : \pi_1\mathrm{Orb} (O_{p,p,\infty}) \to \pi_1\mathrm{Orb} (O_{p,p,\infty}) 
$$
has the property $\bar{a}_* \circ \theta = \theta$ (where $\theta$ is the monodromy of the piece $O_{p,p,\infty}$).

However the monodromies $\theta$ for the $P_i$ appearing in the boundary of strata of type $5$ do not allow such properties. 
\end{proof}

We remark that this theorem provides a number of one dimensional completely isolated strata. Via a straightforward calculation this number can be shown to be quadratic in $p$ but note that they are not necessarily distinct in $\mathcal{M}_{g}$ as several could correspond to the same strata. 

\addcontentsline{toc}{section}{References}
\bibliographystyle{Hugo}
\def\cprime{$'$}

{\em Addresses:}\\
Antonio F. Costa, Departamento de Matematicas Fundamentales, Facultad de Ciencias, Senda del rey, 9, UNED, Madrid, Spain\\
Milagros Izquierdo, Department of Mathematics, University of Linkoping, Sweden\\
Hugo Parlier, Department of Mathematics, University of Fribourg, Switzerland \\
{\em Email:}
\href{mailto:acosta@mat.uned.es}{acosta@mat.uned.es}
 \href{mailto:milagros.izquierdo@liu.se}{milagros.izquierdo@liu.se}
 \href{mailto:hugo.parlier@unifr.ch}{hugo.parlier@unifr.ch}\\

\end{document}